\documentclass[12pt, reqno]{amsart}
\usepackage{amsmath, amsthm, amscd, amsfonts, amssymb, graphicx, color}
\usepackage[bookmarksnumbered, colorlinks, plainpages]{hyperref}

\newtheorem{theorem}{\bf{Theorem}}[section] 
\newtheorem{lemma}[theorem]{\bf{Lemma}}     
\newtheorem{example}[theorem]{\bf{Example}}
\newtheorem{corollary}[theorem]{\bf{Corollary}}
\newtheorem{proposition}[theorem]{\bf{Proposition}}
\newtheorem{definition}[theorem]{\bf{Definition}}

\begin{document}

\title[Characterization of mixed $n$-Jordan and pseudo $n$-Jordan homomorphisms]{Characterization of mixed $n$-Jordan homomorphisms and pseudo $n$-Jordan homomorphisms}

\author[M. Neghabi, A. Bodaghi and  A. Zivari-Kazempour]{Masoumeh Neghabi, Abasalt Bodaghi and Abbas Zivari-Kazempour}

\address{Department of Mathematics, South Tehran
Branch, Islamic Azad University, Tehran, Iran}
\email{m.neghabi114@gmail.com}
\address{Department of Mathematics, Garmsar
Branch, Islamic Azad University, Garmsar, Iran}
\email{abasalt.bodaghi@gmail.com}

\address{Department of Mathematics, Ayatollah Borujerdi University, Borujerd, Iran.}
\email{zivari@abru.ac.ir,\ zivari6526@gmail.com}



\maketitle

\begin{abstract}
In this article, a new notion of $n$-Jordan homomorphism namely the mixed $n$-Jordan homomorphism is introduced. It is proved that how a mixed $(n+1)$-Jordan homomorphism can be a mixed $n$-Jordan homomorphism and vice versa. By means of some examples, it is shown that the mixed $n$-Jordan homomorphisma are different from the $n$-Jordan homomorphisms and the pseudo $n$-Jordan homomorphisms. As a consequence, it shown that every mixed Jordan homomorphism from Banach algebra $\mathcal{A}$ into commutative semisimple Banach algebra $\mathcal{B}$ is automatically continuous. Under some mild conditions, every unital pseudo $3$-Jordan homomorphism can be a homomorphism.
\end{abstract}

\section{Introduction and Preliminaries}

Let $\mathcal{A}$ and $\mathcal{B}$ be complex Banach algebras and
$\varphi:\mathcal{A}\longrightarrow\mathcal{B}$ be a linear map. Then, $\varphi$ is called an {\it $n$-homomorphism} if for all $a_1, a_2, \cdots, a_n\in\mathcal{A}$,
$$
\varphi(a_1a_2\cdots a_n)=\varphi(a_1)\varphi(a_2)\cdots\varphi(a_n).
$$
The concept of an $n$--homomorphism was studied
for complex algebras in \cite{Moslehian} and \cite{Hejazian}. A linear map $\varphi$ between Banach algebras $\mathcal{A}$ and $\mathcal{B}$ is called an $n$-Jordan homomorphism if
$$
\varphi(a^n)=\varphi(a)^n, \ \ \ \ \ a\in \mathcal{A}.
$$
This notion was introduced by Herstein in \cite{Herstein}. A $2$-homomorphism ($2$-Jordan homomorphism) is called simply a homomorphism (Jordan homomorphism). It is clear that every $n$-homomorphism is an $n$-Jordan homomorphism, but in general the converse is false. Indeed, it was Ancochea \cite{anc} who firstly studied the connection of Jordan homomorphisms and homomorphisms. The results of Ancochea were generalized and extended in several ways in \cite{Jacobson} and \cite{kap}. There are plenty of known examples of $n$-Jordan homomorphism which are not $n$-homomorphism. For $n=2$, it is proved in \cite{Jacobson} that some Jordan homomorphism on the polynomial rings can not be homomorphism. Also, each homomorphism is an $n$-homomorphism for every $n\geq 2$, but the converse does not hold in general. For instance, if $h:\mathcal A\longrightarrow \mathcal B$ is a homomorphism, then $g:= -h$ is a 3-homomorphism which is not a homomorphism \cite{Moslehian}. Herstein \cite{Herstein} proved the following result.

\begin{theorem}  If $\varphi$ is a Jordan homomorphism of a ring $\mathcal R$ onto a prime ring $\mathcal R'$ of characteristic
different from $2$ and $3$, then either $\varphi$ is a homomorphism or an anti-homomorphism. 
\end{theorem}

After that, Zelazko in  \cite{Zelazko} presented the upcoming  result (see also \cite{Miura}).
\begin{theorem}\label{ze}
Suppose that $\mathcal{A}$ is a Banach algebra, which need not be commutative, and suppose that
$\mathcal{B}$ is a semisimple commutative Banach algebra. Then each Jordan
homomorphism $\varphi:\mathcal{A}\longrightarrow \mathcal{B}$ is a homomorphism.
\end{theorem}
 This result has been proved by the third author in \cite{zivari1} for $3$-Jordan homomorphism with the extra condition that the Banach algebra $\mathcal{A}$ is unital. In other words, he presented the next theorem.

\begin{theorem}\label{ttt}
Suppose that $\mathcal{A}$ is a unial Banach algebra, which need not be commutative, and suppose that
$\mathcal{B}$ is a semisimple commutative Banach algebra. Then each 3-Jordan
homomorphism $\varphi:\mathcal{A}\longrightarrow \mathcal{B}$ is a 3-homomorphism.
\end{theorem}

After that, An \cite{An} extended the above theorem for all $n\in \mathbb{N}$ in \cite{An} and showed that for unital ring $\mathcal A$ and ring $\mathcal B$ with char($\mathcal B$)$>n$,  every $n$-Jordan
homomorphism from $\mathcal A$ into $\mathcal B$ is an $n$-homomorphism ($n$-anti-homomorphism) provided that every Jordan homomorphism from $\mathcal A$ into $\mathcal B$ is a homomorphism (anti-homomorphism). Recently, the second author and \.{I}nceboz extended Theorem \ref{ze} for $n\in\{3,4\}$ in \cite{Bodaghi1} without the Banach algebra $\mathcal A$ is that of being unital by considering an extra condition on the mapping $\varphi$ as
 $$
 \varphi(a^2b)=\varphi(ba^2), \  \  a,b\in \mathcal{A}.
 $$

Some significant results concerning Jordan homomorphisms and
their automatic continuity on Banach algebras obtained by the  third author in \cite{zivari}, \cite{zivari2} and \cite{zivari3}. For the commutative case, Lee in \cite{Lee} and Gselmann in \cite{Gselmann} every $n$-Jordan homomorphism between two commutative Banach algebras is an $n$-homomorphism where $n$ is an arbitrary and fixed positive integer. Later, this problem solved in \cite{Bodaghi} based on the property of the Vandermonde matrix, which is different from the methods that are used in \cite{Gselmann} and \cite{Lee}.

 Let $\mathcal{A}$ and $\mathcal{B}$ be rings (algebras), and let
$\mathcal{B}$ be a right [left] $\mathcal{A}$-module. Then a linear map $\varphi:\mathcal{A}\longrightarrow \mathcal{B}$ is saied to be
pseudo $n$-Jordan homomorphism if there exiest an element $w\in \mathcal{A}$ such that
 $$
 \varphi(a^nw)=\varphi(a)^n\cdot w, \,\, [\varphi(a^nw)=w\cdot \varphi(a)^n]\quad (a\in \mathcal{A}).
 $$
 The element $w$ is called Jordan coefficient of $\varphi$. The concept of pseudo $n$-Jordan homomorphism was introduced and studied  by Ebadian et al., in \cite{Ebadian}. They also investigated the automatic continuity such homomorphisms on commutative $C^*$-algebras and semisimple (non unital) Banach algebras.

In section 2, we introduce the notion of mixed $n$-Jordan homomorphism on algebras. We prove that 
every $3$-Jordan homomorphism $\varphi$ from algebra $\mathcal{A}$ into $\varphi$-commutative algebra $\mathcal{B}$ such that $\varphi(ab-ba)=0$ is a mixed Jordan homomorphism. Also, we discuss the automatic continuity of mixed Jordan homomorphisms. We show that under which conditions  a mixed $(n+1)$-Jordan homomorphism is a mixed $n$-Jordan homomorphism and vice versa. We prove that every $n$-Jordan homomorphism on non-commutative Banach algebras is a $n$-homomorphism with different conditions as in \cite[Theorem 2.4]{An} and \cite[Theorem 2.2]{Bodaghi}.
It is of interest to know whether the converse of \cite[Theorem 2.3]{Ebadian} holds. In the last section, we answer to this question. In fact, we show that every unital pseudo $(n+1)$-Jordan homomorphismpseudo $n$-Jordan homomorphism with the same Jordan coefficient.



\section{Mixed $n$-Jordan homomorphisms}

Let $\mathcal{A}$ and $\mathcal{B}$ be complex algebras and
$\varphi:\mathcal{A}\longrightarrow\mathcal{B}$ be a linear map. Then $\varphi$ is called an {\it mixed $n$-Jordan homomorphism} if for all $a,b\in\mathcal{A}$,
$$
\varphi(a^nb)=\varphi(a)^n\varphi(b).
$$

A mixed $2$-Jordan homomorphism is said to be {\it mixed Jordan homomorphism}. Clearly, every $n$-homomorphism is an mixed $(n-1)$-Jordan homomorphism for $n\geq 3$, and every mixed $n$-Jordan homomorphism is $(n+1)$-Jordan homomorphism but the converse is not true in general. The following example illustrates this fact.
\begin{example}
\emph {Let $\mathcal{A}=\left\{
\begin{bmatrix}
X & 0 \\
0 &  Y
\end{bmatrix}:\ \ \  X,Y\in M_2(\mathbb{C})
\right\},
$
and define $\varphi:\mathcal{A}\longrightarrow \mathcal{A}$ by
$
\varphi\left(\begin{bmatrix}
X & 0 \\
0 &  Y
\end{bmatrix}\right)=
\begin{bmatrix}
X & 0 \\
0 &  Y^T
\end{bmatrix}$, where $Y^T$ is the transpose of matrix $Y$. Then, for all $U\in \mathcal{A}$, we have}
$$
\varphi(U^n)=\varphi(U)^n=
\begin{bmatrix}
X^n & 0 \\
0 &   (Y^n)^{T}
\end{bmatrix}.
$$
\emph {Thus, $\varphi$ is $n$-Jordan homomorphism, but $\varphi$ is not mixed $(n-1)$-Jordan homomorphism. In fact, for}

$$
U=\begin{bmatrix}
X & 0 \\
0 &  Y
\end{bmatrix},\ \
V=\begin{bmatrix}
A & 0 \\
0 &  B
\end{bmatrix}.
$$
\emph {We have $(Y^{(n-1)}B)^T\neq (Y^T)^{(n-1)} B^T.$ Therefore, $\varphi(U^{(n-1)}V)\neq \varphi(U)^{(n-1)}\varphi(V)$.}
\end{example}


A left ideal $\mathcal I$ of an algebra $\mathcal A$ is a modular left ideal if there exists $u\in \mathcal A$ such that $\mathcal A(e_{\mathcal A}-u)\subseteq \mathcal I$, where $\mathcal A(e_{\mathcal A}-u)=\{x-xu: x\in \mathcal A\}$. The Jacobson radical Rad ($\mathcal A$) of $\mathcal A$ is the intersection of all maximal modular left ideals of $\mathcal A$. An algebra $\mathcal A$ is called {\it semisimple} whenever its Jacobson radical Rad ($\mathcal A$) is trivial. Also, every $C^*$-algebra is semisimple.

Let $n$ be an integer $n\geq 2$. Recall that  an associative ring $\mathcal R$ is of
characteristic not $n$ if $na = 0$ implies $a=0$ for every $a\in \mathcal R$, and $\mathcal R$ is of
characteristic greater than $n$ if $n!a=0$ implies $a=0$ for every $a\in \mathcal R$.

Here, we bring some trivial observations are as follows:

$\bullet$ It is known that every mixed $n$-Jordan homomorphism is $(n+1)$-Jordan homomorphism, and so by \cite[Theorem 2.2]{Bodaghi}, every mixed $n$-Jordan homomorphism $\varphi$ between commutative algebras $\mathcal{A}$ and $\mathcal{B}$ is $(n+1)$-homomorphism.

$\bullet$ Let $\mathcal{A}$ be a unital Banach algebra and $\mathcal B$ be a Banach algebra with char($\mathcal B$)$>n$. By \cite[Theorem 2.4]{An}, every mixed $n$-Jordan homomorphism from $\mathcal A$ into $\mathcal B$ is an $(n+1)$-homomorphism if every Jordan homomorphism from $\mathcal A$ into $\mathcal B$ is a homomorphism. So, under such assumptions and that $\mathcal{B}$ is a semisimple commutative Banach algebra, every surjective mixed $n$-Jordan homomorphism $\varphi:\mathcal{A}\longrightarrow\mathcal{B}$ is automatically continuous by \cite[Corollary 2.5]{An}.


\begin{lemma}\label{2}
Let $\varphi$ be a mixed Jordan homomorphism between algebras $\mathcal{A}$ and $\mathcal{B}$. Then
\begin{enumerate}
\item [(i)] $\varphi$ is mixed $(2n)$-Jordan homomorphism for all $n\in \mathbb{N}$.

\item [(ii)] if $\mathcal{A}$ is unital, then $\varphi(x)\varphi(e)=\varphi(e)\varphi(x)$, where $e$ is the idntity of $\mathcal{A}$.

\item [(iii)] the mapping $\psi(x):=\varphi(x)\varphi(e)$ is a homomorphism.
\end{enumerate}
\end{lemma}
\begin{proof}
Suppose that $\varphi$ is mixed Jordan homomorphism. Then for all $a,b\in \mathcal{A}$,
\begin{equation}\label{R1}
\varphi(a^2b)=\varphi(a)^2\varphi(b).
\end{equation}
Replacing $b$ by $a^2b$ in (\ref{R1}), gives
$$
\varphi(a^4b)=\varphi(a)^2\varphi(a^2b)=\varphi(a)^4\varphi(b).
$$
Thus, $\varphi$ is mixed $4$-Jordan homomorphism. This argument can be repeated to achieve the desired result of part (i).\\

For the part (ii), replacing $b$ by $b+e$ in (\ref{R1}), we have
\begin{equation}\label{R3}
2\varphi(ab)=[\varphi(a)\varphi(e)+\varphi(e)\varphi(a)]\varphi(b).
\end{equation}
for all $a,b\in \mathcal{A}$. Switching $b$ by $e$ in (\ref{R3}),  we find
\begin{equation}\label{R4}
2\varphi(a)=\varphi(a)\varphi(e)^2+\varphi(e)\varphi(a)\varphi(e).
\end{equation}
for all $a\in \mathcal{A}$. Multiplying $\varphi(e)$ from the left in (\ref{R4}), and using the equality $\varphi(e)^3=\varphi(e)$, we get
\begin{equation}\label{R5}
2\varphi(a)\varphi(e)=\varphi(a)\varphi(e)+\varphi(e)\varphi(a)\varphi(e)^2.
\end{equation}
for all $a\in \mathcal{A}$. Similarly,
\begin{equation}\label{R6}
2\varphi(e)\varphi(a)=\varphi(e)\varphi(a)\varphi(e)^2+\varphi(e)^2\varphi(a)\varphi(e).
\end{equation}
The relations (\ref{R5}) and (\ref{R6}) necessitate that
\begin{equation}\label{R7}
2[\varphi(e)\varphi(a)-\varphi(a)\varphi(e)]=\varphi(e)^2\varphi(a)\varphi(e)-\varphi(a)\varphi(e).
\end{equation}
Multiplying $\varphi(e)$ from the right in (\ref{R7}), we arrive at
$$
2\varphi(e)[\varphi(e)\varphi(a)-\varphi(a)\varphi(e)]=\varphi(e)\varphi(a)\varphi(e)-\varphi(e)\varphi(a)\varphi(e)=0.
$$
Thus,
\begin{equation}\label{R9}
\varphi(e)^2\varphi(a)=\varphi(e)\varphi(a)\varphi(e).
\end{equation}
for all $a\in \mathcal{A}$. Once more, multiplying $\varphi(e)$ from the right in (\ref{R9}), we get
\begin{equation}\label{R10}
\varphi(e)\varphi(a)=\varphi(e)^2\varphi(a)\varphi(e).
\end{equation}
It now follows from (\ref{R7}) and (\ref{R10}) that $\varphi(e)\varphi(a)=\varphi(a)\varphi(e),$
for all $a\in \mathcal{A}$. This completes the proof of (ii).\\

By the part (ii) and (\ref{R3}) we see that $\varphi(ab)=\varphi(a)\varphi(b)\varphi(e),$ for all $a,b\in \mathcal{A}$. 
\end{proof}


For certain calculations, we use the notation $[a,b]=ab-ba$ which is called the Lie product of $a$ and $b$.
Let $\varphi:\mathcal{A}\longrightarrow \mathcal{B}$ be a map between Banach algebras $\mathcal{A}$ and $\mathcal{B}$. Then,  we say that $\mathcal{B}$ is $\varphi$-commutative if for all $a,b\in \mathcal{A}$, $[\varphi(a),\varphi(b)]=0$.

Note that every commutative Banach algebra is $I$-commutative, where $I$ is the identity map. 

\begin{example}\label{ex1}
\emph{(i)  Consider the Banach algebras}
$
\mathcal{A}= \left\{
 \begin{bmatrix}
u & a & b \\
0 &  0 & c\\
0 &  0 & 0
\end{bmatrix}
:\ \ \  u,a,b,c\in\mathbb{C}
\right\},$
$
\mathcal{B}= \left\{
 \begin{bmatrix}
a & b \\
0 &  0
\end{bmatrix}
:\ \ \  a,b\in\mathbb{C}
\right\}
$
\emph{with the usual sum and product. Define the linear map 
$\varphi: \mathcal{A} \longrightarrow \mathcal{B}$ by}
\begin{equation*}
\varphi\left(
 \begin{bmatrix}
u & a & b \\
0 &  0 & c\\
0 &  0 & 0
\end{bmatrix}\right)
= \begin{bmatrix}
u & 0 \\
0 &  0
\end{bmatrix}.
\end{equation*}
\emph{Then, $\mathcal{B}$ is non-commutative Banach algebra, but it is $\varphi$-commutative.}

\emph{(ii)  Let $\mathcal{A}$ as the part (i). Consider $\varphi:\mathcal{A}\longrightarrow \mathbb{C}$ defined via $\varphi\left(\begin{bmatrix}
a & b & c\\
0 & 0 & u\\
0 & 0 & 0
\end{bmatrix}\right)=u$. Then, $\varphi(XY)=\varphi(YX)$ and  $[\varphi(X),\varphi(Y)]=0$ for all $X,Y\in \mathcal{A}$. We see that $\varphi$ is not $3$-Jordan homomorphism and so it is neither mixed Jordan homomorphism nor $3$-homomorphism.}
\end{example}


\begin{lemma}\label{lem1}
Let $\varphi$ be an $n$-Jordan homomorphism from unital Banach algebra $\mathcal{A}$ into $\varphi$-commutative Banach algebra $\mathcal{B}$.
Then, for all $a\in \mathcal{A}$,
$$
\varphi(a)=\varphi(e)^{n-1}\varphi(a)=\varphi(a)\varphi(e)^{n-1}.
$$
\end{lemma}
\begin{proof}
Let $\varphi$ be a Jordan homomorphism. Then
\begin{equation}\label{ae1}
\varphi((a+3)^2-2(a+2)^2+a^2)=\varphi(a+3)^2-2\varphi(a+2)^2+\varphi(a)^2.
\end{equation}
for all $a\in \mathcal{A}$. Let $e$ be the identity of $\mathcal{A}$. By assumption $\varphi(e)=\varphi(e)^2$. So, (\ref{ae1}) gives
\begin{equation}\label{ae2}
2\varphi(a)=\varphi(a)\varphi(e)+\varphi(e)\varphi(a).
\end{equation}
for all $a\in \mathcal{A}$. On the other hand, 
\begin{equation}\label{ae3}
[\varphi(a),\varphi(e)]=\varphi(a)\varphi(e)-\varphi(e)\varphi(a)=0,\ \ \ \ (a\in \mathcal{A}).
\end{equation}
It follows from (\ref{ae2}) and (\ref{ae3}) that $\varphi(a)=\varphi(e)\varphi(a)=\varphi(a)\varphi(e)$,
for all $a\in \mathcal{A}$.
Now, assume that $n=3$. Then 
\begin{equation}\label{ae4}
\varphi((a+2)^3-2(a+1)^3+a^3)=\varphi(a+2)^3-2\varphi(a+1)^3+\varphi(a)^3.
\end{equation} 
for all $a\in \mathcal{A}$. Since $\varphi(e)=\varphi(e)^3$,  the relation (\ref{ae4}) implies that
\begin{equation}\label{ae5}
3\varphi(a)=\varphi(a)\varphi(e)^2+\varphi(e)^2\varphi(a)+\varphi(e)\varphi(a)\varphi(e).
\end{equation}
for all $a\in \mathcal{A}$. By the $\varphi$-commutativity of $\mathcal{B}$, we have
\begin{equation}\label{ae6}
\varphi(a)\varphi(e)=\varphi(e)\varphi(a),\ \ \ \ (a\in \mathcal{A}).
\end{equation}
for all $a\in \mathcal{A}$. Plugging (\ref{ae5}) into (\ref{ae6}), we find
$\varphi(a)=\varphi(e)^2\varphi(a)=\varphi(a)\varphi(e)^2$. Similarily, one can obtain the result for all $n\geq 4$.
\end{proof}


\begin{lemma}\label{lem2}
 Let $\mathcal{A}$ be unital Banach algebra with unit $e$, and $\varphi:\mathcal{A}\longrightarrow \mathcal{B}$ be a 
$n$--Jordan homomorphism. Then
\begin{equation*}
\varphi(a^2)=\varphi(a)^2\varphi(e)^{n-2}\qquad (a\in \mathcal{A}).
\end{equation*}
\end{lemma}
\begin{proof}
See the proof of Theorem 2.4 from \cite{An}.
\end{proof}


As mentioned that every $n$-Jordan homomorphism between commutative Banach algebras $\mathcal{A}$ and $\mathcal{B}$ is an $n$-homomorphism \cite[Theorem 2.2]{Bodaghi}. In view of the proof of this theorem, we see the same result holds by the weaker condition on $\mathcal{B}$, as $\varphi$-commutativity. However, in the next result, we show that under some conditions every $n$-Jordan homomorphism on non-commutative Banach algebras can be an $n$-homomorphism.

\begin{theorem}\label{6}

Every $n$-Jordan homomorphism $\varphi:\mathcal{A}\longrightarrow \mathcal{B}$ from unital Banach algebra $\mathcal{A}$ into $\varphi$-commutative Banach algebra $\mathcal{B}$ satisfiying the following condition  is a $n$-homomorphism.

\begin{equation}\label{a1}
\varphi(x^2)=0 \Longrightarrow \varphi(x)=0, \   \   \   (x\in \mathcal{A}).
\end{equation}
\end{theorem}
\begin{proof}
Define a mapping $\psi:\mathcal{A}\longrightarrow\mathcal{B}$ through
$$
\psi(a):=\varphi(a)\varphi(e)^{n-2},\    \     \    (a\in \mathcal{A}).
$$
By Lemma \ref{lem2}, $\psi$ is Jordan homomorphism and $\mathcal{B}$ is $\psi$-commutative. Also, Lemma \ref{lem1} necessitates that $\psi$ satisfies the condition (\ref{a1}). We wish to show that $\psi$ is a homomorphism. For all $a,b\in \mathcal{A}$, we have
\begin{equation}\label{a8}
\psi(ab+ba)=\psi(a)\psi(b)+\psi(b)\psi(a).
\end{equation}
The $\psi$-commutativity of $\mathcal{B}$ implies that the equality (\ref{a8}) converts to
\begin{equation}\label{a9}
\psi(ab+ba)=2\psi(a)\psi(b).
\end{equation}
The mapping $\psi$ is a Jordan homomorphism and hence
\begin{equation}\label{a10}
\psi([a,b]^2)=[\psi(a),\psi(b)]^2=0.
\end{equation}
Since $\psi$ satisfies the condition (\ref{a1}), it concludes from  (\ref{a10}) that
\begin{equation}\label{a11}
\psi(ab-ba)=0.
\end{equation}
Plugging (\ref{a9}) into (\ref{a11}), we have $\psi(ab)=\psi(a)\psi(b)$. By Lemma \ref{lem1}, we get
\begin{equation}\label{a17}
\psi(a)\varphi(e)=\varphi(a)\varphi(e)^{n-1}=\varphi(a).
\end{equation}
It follows from Lemma \ref{lem1} and the relation (\ref{a17}) that
\begin{eqnarray}
\nonumber \varphi(a_1a_2\cdots a_n) &=& \psi(a_1a_2\cdots a_n)\varphi(e)\\
\nonumber &=& \psi(a_1)\psi(a_2)\cdots \psi(a_n)\varphi(e)\\
\nonumber &=& (\varphi(a_1)\varphi(e)^{n-2})(\varphi(a_2)\varphi(e)^{n-2})\cdots(\varphi(a_n)\varphi(e)^{n-2})\varphi(e)\\
\nonumber &=& \varphi(a_1)\varphi(a_2)\cdots\varphi(a_n)\varphi(e)^{(n-1)^2}\\
\nonumber &=& \varphi(a_1)\varphi(a_2)\cdots\varphi(a_n).
\end{eqnarray}
Thus, $\varphi$ is $n$--homomorphism.
\end{proof}


\begin{theorem}\label{40}
Let $\varphi$ be a $3$-Jordan homomorphism from algebra $\mathcal{A}$ into $\varphi$-commutative algebra $\mathcal{B}$ such that $\varphi([a,b])=0$ for any $a,b\in \mathcal{A}$. Then, $\varphi$ is mixed Jordan homomorphism.
\end{theorem}
\begin{proof}
By assumption 
\begin{equation}\label{P1}
\varphi(x^3)=\varphi(x)^3,
\end{equation}
for all $x\in \mathcal{A}$. Replacing $x$ by $a+b$ in (\ref{P1}), we obtain
\begin{equation}\label{P2}
\varphi(aba+ba^2+a^2b+b^2a+ab^2+bab)=3[\varphi(a)^2\varphi(b)+\varphi(a)\varphi(b)^2].
\end{equation}
Switching $b$ by $-b$ in (\ref{P2}), we get
\begin{equation}\label{P3}
\varphi(-aba-ba^2-a^2b+b^2a+ab^2+bab)=3[-\varphi(a)^2\varphi(b)+\varphi(a)\varphi(b)^2],
\end{equation}
for all $a,b\in \mathcal{A}$. The equalities (\ref{P2}) and (\ref{P3}) show that
\begin{equation}\label{P4}
\varphi(b^2a+ab^2+bab)=3\varphi(a)\varphi(b)^2, \   \   (a,b\in \mathcal{A}).
\end{equation}
 Since $\varphi([a,b])=0$, we have
\begin{equation}\label{P5}
\varphi(b^2a)=\varphi(ab^2)=\varphi(bab),
\end{equation}
for all $a,b\in \mathcal{A}$. It follows from (\ref{P4}) and (\ref{P5}) that
$\varphi(b^2a)=\varphi(b)^2\varphi(a)$, for all $a,b\in \mathcal{A}$. Therefore, $\varphi$ is mixed Jordan homomorphism.
\end{proof}


As a consequence of Theorem \ref{40}, we have the next result.
\begin{corollary}
Let $\varphi$ be a mapping from algebra $\mathcal{A}$ into $\varphi$-commutative Banach algebra $\mathcal{B}$, such that $\varphi([a,b])=0$ for all $a,b\in \mathcal{A}$. Suppose $\delta>0$ and $\varphi$ satisfy
\begin{equation}\label{pp6}
|\varphi(a^2)-\varphi(a)^2|\leq \delta,
\end{equation}
then $\varphi$ is mixed Jordan homomorphism.
\end{corollary}


The following theorem is a well-known result, due to
$\check{\text{S}}$ilov, concerning the automatic continuity of homomorphisms between Banach
algebras.

\begin{theorem}\label{SI}
 Let $\mathcal{A}$ and $\mathcal{B}$ be Banach algebras
such that $\mathcal{B}$ is commutative and semisimple. Then, every
homomorphism $\varphi:\mathcal{A} \longrightarrow \mathcal{B}$ is automatically continuous.
\end{theorem}


 In 1967, B. E. Johnson proved that if $\varphi: \mathcal{A} \longrightarrow \mathcal{B}$ is a surjective
homomorphism between a Banach algebra $\mathcal{A}$ and a semisimple Banach
algebra $\mathcal{B}$, then $\varphi$ is automatically continuous and then the Johnson's result extended to $n$-homomorphism in \cite{Gordji1}. One
may refer to \cite{Moslehian} for automatic continuity of $3$-homomorphism.\\


We say that a linear map $\varphi: \mathcal{A} \longrightarrow \mathcal{B}$ is a {\it co-Jordan homomorphism} if for
all $a\in \mathcal{A}$, $\varphi(a^2)=-\varphi(a)^2$. For example, the function $\varphi: \mathbb{R} \longrightarrow \mathbb{R}$ defined by $\varphi(a)=-a$ is a co-Jordan homomorphism. Here, we show that Theorem \ref{SI} holds for mixed Jordan homomorphisms.

\begin{theorem}\label{4}
Let $\varphi$ be a mixed Jordan homomorphism from Banach algebra $\mathcal{A}$ into $\mathbb{C}$. Then, $\varphi$ is automatically continuous.
\end{theorem}
\begin{proof}
Suppose that there exist $x_0\in \mathcal{A}$ such that $\|x_0\|<1$ and $\varphi(x_0)=1$. Take $y=\sum_{n=1}^\infty x_0^{n}$. Then, $x_0+x_{0}^2y=y-x_{0}^2,$ and so
$$
1+\varphi(y)=\varphi(x_0)+\varphi(x_{0})^2\varphi(y)=\varphi(x_0+x_{0}^2y)=\varphi(y)-\varphi(x_{0}^2).
$$
Thus, $\varphi(x_{0}^2)=-1$. Since $\varphi$ is mixed Jordan homomorphism, we have 
\begin{equation}\label{RE1}
\varphi(a^2b)=\varphi(a)^2\varphi(b).
\end{equation}
for all $a,b\in \mathcal{A}$. Replacing $a$ by $u+v$ in (\ref{RE1}), we get
\begin{equation}\label{RE2}
\varphi((uv+vu)b)=2\varphi(u)\varphi(v)\varphi(b).
\end{equation}
for all $a,u,v\in \mathcal{A}$. Interchanging $b$ by $x_{0}^2$ in (\ref{RE2}), we obtain
\begin{equation}\label{RE3}
\varphi((uv+vu)x_{0}^2)=2\varphi(u)\varphi(v)\varphi(x_{0}^2)=-2\varphi(u)\varphi(v).
\end{equation}
for all $a,u,v\in \mathcal{A}$. Substituting  $(u,v)$ into $(u^2,x_{0}^2)$ in (\ref{RE3}), we arrive at
\begin{equation}\label{RE4}
\varphi(u^2x_{0}^4)+\varphi(x_{0}^2u^2x_{0}^2)=-2\varphi(u^2)\varphi(x_{0}^2)=2\varphi(u^2).
\end{equation}
for all $u\in \mathcal{A}$. Also, $\varphi(x_{0}^4)=\varphi(x_{0})^2\varphi(x_{0}^2)=-1$, and
$$
varphi(x_{0}^2u^2x_{0}^2)=\varphi(x_{0})^2\varphi(u^2x_{0}^2)=\varphi(x_{0})^2\varphi(u)^2\varphi(x_{0}^2)=-\varphi(u)^2,
$$
for all $u\in \mathcal{A}$. So by (\ref{RE4}), we have $\varphi(u)^2=-\varphi(u^2)$ for all $u\in \mathcal{A}$. 
Hence, $\varphi$ is co-Jordan homomorphism and it is continuous by \cite[Proposition 2.1]{zivari3}. If there is no $x_0\in \mathcal{A}$ such that $\|x_0\|<1$ and $\varphi(x_0)=1$, then for all $x\in \mathcal{A}$ with $\|x\|<1$ we have $|\varphi(x)|\leq1$. Therefore, $\varphi$ is continuous.
\end{proof}


The upcoming corollary is a direct consequence of Theorem \ref{4}.
\begin{corollary}
Let $\varphi$ be a mixed Jordan homomorphism from Banach algebras $\mathcal{A}$ into commutative semisimple Banach algebra $\mathcal{B}$. Then $\varphi$ is automatically continuous.
\end{corollary}
\begin{proof}
Let $(a_n)\subseteq \mathcal{A}$, $a_n\longrightarrow 0$ and $\varphi(a_n)\longrightarrow b$. Suppose that $h\in\mathfrak{M}(\mathcal{B})$, where $\mathfrak{M}(\mathcal{B})$ is the maximal ideal space of $\mathcal{B}$. Then, $h\circ \varphi$ is a mixed Jordan homomorphism and so it is continuous by Theorem \ref{4}. Thus,
$$
h(b)=\lim_n h (\varphi(a_n))=\lim_n h\circ \varphi(a_n)=0.
$$
Since $\mathcal{B}$ is semisimple, we have $b=0$, and thus $\varphi$  is continuous by the close graph theorem.
\end{proof}

The next result is the same as Theorem \ref{SI} for $3$-homomorphism. 
\begin{corollary}
Let $\varphi$ be a $3$-homomorphism from Banach algebras $\mathcal{A}$ into commutative semisimple Banach algebra $\mathcal{B}$. Then $\varphi$ is automatically continuous.
\end{corollary}


A linear map $\varphi$ between unital Banach algebras $\mathcal{A}$ and $\mathcal{B}$ is called \textit{unital} if $\varphi(e)=e'$, where $e$ and $e'$ are the unit element of $\mathcal{A}$ and $\mathcal{B}$, respectively.

\begin{theorem}\label{T1}
Every unital mixed $(n+1)$-Jordan homomorphism $\varphi:\mathcal{A}\longrightarrow\mathcal{B} $ is a homomorphism.
\end{theorem}
\begin{proof}
Let $\varphi$ be a mixed $(n+1)$-Jordan homomorphism. Then we have 
\begin{equation}\label{RE5}
\varphi(a^{n+1}b)=\varphi(a)^{n+1}\varphi(b),
\end{equation}
for all $a,b\in \mathcal{A}$. Since $\varphi$ is unital, then with $b=e$ in (\ref{RE5}) we get
\begin{equation}\label{RE6}
\varphi(a^{n+1})=\varphi(a)^{n+1},
\end{equation}
for all $a\in \mathcal{A}$. Replacing $a$ by $a+\lambda e$ in (\ref{RE6}), where $\lambda$ is a complex number, and compare powers of $\lambda$, we arrive at
\begin{equation}\label{RE7}
\varphi(a^2)=\varphi(a^2),
\end{equation}
hence $\varphi$ is a Jordan homomorphism. Interchanging $a$ by $a+\lambda e$ in (\ref{RE5}), we botain
\begin{equation}\label{RE8}
\varphi(a^{n+1}b)=(\varphi(a)+\lambda \varphi(e))^{n+1}\varphi(b).
\end{equation}
Comparing powers of $\lambda^{n}$ in (\ref{RE8}) and using $\varphi(e)=e'$, one deduce that $\varphi$ is a homomorphism.
\end{proof}


The next corollary follows immediately from Theorem \ref{T1}.
\begin{corollary}\label{C01}
Every unital mixed $(n+1)$-Jordan homomorphism $\varphi:\mathcal{A}\longrightarrow\mathcal{B} $ is a mixed $n$-Jordan homomorphism. 
\end{corollary}


The following example shows the condition being unital for Banach algebras $\mathcal{A}$ and $\mathcal{B}$ in Theorem \ref{T1} is essential.
\begin{example}
Let
\begin{equation*}
\mathcal{A}= \left\{
 \begin{bmatrix}
0 & a & b \\
0 &  0 & c\\
0 &  0 & 0
\end{bmatrix}
:\ \ \  a,b,c\in\mathbb{R}
\right\},
\end{equation*}
and define $\varphi: \mathcal{A} \longrightarrow \mathcal{A}$ via
$$
\varphi\left(\begin{bmatrix}
0 & a & b \\
0 &  0 & c\\
0 &  0 & 0
\end{bmatrix}\right)
=
\begin{bmatrix}
0 & a & 0 \\
0 &  0 & c\\
0 &  0 & 0
\end{bmatrix}.
$$
Then,  $\varphi(X^2)\neq\varphi(X)^2$, for all $X\in \mathcal{A}$. Hence, $\varphi$ is not Jordan homomorphism, and so it is not  homomorphism.
But for all $n\geq 3$ and for all $X,Y\in \mathcal{A}$, we have $\varphi(X^nY)=\varphi(X)^n\varphi(Y)$. Therefore, $\varphi$ is mixed $n$-Jordan homomorphism for all $n\geq 3$.
\end{example}

\begin{corollary}
Let $\varphi:\mathcal{A}\longrightarrow\mathcal{B} $ be a unital mixed $n$-Jordan homomorphism. Then, $\varphi$ is continuous under one of the following conditions.
\begin{enumerate}
\item[(i)]  $\mathcal{B}$ is semisimple and commutative.

\item[(ii)] $\mathcal{B}$ is semisimple and $\varphi$ is surjective.

\item[(iii)] $\mathcal{B}$ is $C^*$-algebra and $\varphi$ is surjective.
\end{enumerate}
\end{corollary}
\begin{proof}
By Corollary \ref{C01}, $\varphi$ is a homomorphism and so it is $n$-Jordan homomorphism. Thus, the result follows from Corollaries 2.9, 2.10 and 2.11 of \cite{zivari3}.
\end{proof}



In the upcoming result, we prove the converse of Theorem \ref{T1} under some conditions. The idea of the proof is taken from \cite[Theorem 2.6]{Ebadian}. We include the proof for the sake of completeness.

\begin{theorem}\label{tht}
Let $\mathcal{A}$, $\mathcal{B}$ be Banach algebras such that $\mathcal{A}$ is unital and $\varphi:A\longrightarrow B$ be mixed $n$-Jordan honmomorphism. Suppose that there exists an idempotent $p$ in $\mathcal{A}$ such that $\varphi (ab)= \varphi (a)\varphi (b)$ for  all  $a,b,x \in \mathcal{A}$ with $ab=px$.  Then, $\varphi(a^{n+1}px)=\varphi(a)^{n+1}\varphi(px)$  for  all $a,x\in \mathcal A$. In particular, $\varphi$ is mixed $(n+1)$-Jordan honmomorphism.
\end{theorem}
\begin{proof}
Let $e$ be a until element  of $\mathcal{A}$ and $a\in\mathcal{A}$. For $\lambda \in \mathbb C$, with $ |\lambda|<1/\|a\|$, $e-\lambda$ is invertible and $(e-\lambda a)^{-1}= \sum_{n=0}^{\infty} \lambda ^n a^n.$ Then 

\begin{align*}
 \varphi(px)&= \varphi ((e-\lambda a)(e-\lambda a)^{-1}px)\\
 &= \varphi(e-\lambda a)\varphi((e-\lambda a)^{-1}px)\\
 &=(\varphi(e)-\lambda \varphi (a)) \varphi\left(\sum_{n=0}^{\infty} \lambda ^n a^n px\right)\\
 &=\varphi (e)\varphi(px)+ \varphi\left( \sum_{n=1}^{\infty}\lambda ^n a^n px\right)-\lambda \varphi(a)\varphi\left( \sum_{n=0}^{\infty} \lambda ^n a^n px\right)\\
 & = \varphi (px) + \sum_{n=1}^{\infty} \lambda^n \varphi (a^npx)-\lambda \varphi(a)\varphi\left(\sum_{n=0}^{\infty} \lambda ^n a^n px\right)\\
 &=\varphi (px) + \sum_{n=0}^{\infty} \lambda ^{n+1}\varphi (a^{n+1}px)-\varphi (a)\sum_{n=0}^{\infty} \lambda ^{n+1}\varphi (a^npx)
\end{align*}
Hence, $\sum_{n=0}^{\infty} \lambda^{n+1}[\varphi (a^{n+1}px)-\varphi(a)\varphi(a^{n}px)]=0$ for $\lambda \in \mathbb C$, with $ |\lambda|<1/\|a\|$. Thus, 
$\varphi(a^{n+1}px)=\varphi(a)\varphi (a^npx)$ for $n=0,1,2,\cdots.$ For $n\geq 1$, we have
$$ \varphi(a)\varphi(a^npx)=\varphi(a)\varphi(a)^n\varphi(px)=\varphi (a)^{n+1}\varphi(px).$$
So, $\varphi(a^{n+1}px)=\varphi(a)^{n+1}\varphi(px)$   for all $a,x\in \mathcal A$. If we take $p=e$, we get $\varphi(a^{n+1}x)=\varphi(a)^{n+1}\varphi(x)$. This finishes the proof.
\end{proof}

We should note that Theorem \ref{tht} is true if ``mixed'' is deleted from it.

In the following result, under certain conditions, we prove that each mixed Jordan homomorphism is a mixed $n$-Jordan homomorphism.
\begin{theorem}\label{T2}
Let $\varphi:\mathcal{A}\longrightarrow\mathcal{B} $ be a unital mixed  Jordan homomorphism. Then,
$\varphi$ is mixed $n$-Jordan homomorphism.
\end{theorem}
\begin{proof}
The proof follows from Corollary \ref{C01}.
\end{proof}

\begin{proposition}
Let $\varphi:\mathcal{A}\longrightarrow \mathcal{B} $ be a linear map such that
\begin{equation}\label{ALM1}
\Vert \varphi(a^nb)-\varphi(a)^n\varphi(b) \Vert \leqslant \delta \Vert a\Vert^n \Vert b \Vert,
\end{equation}
for all $a,b\in \mathcal{A}$ and for some $\delta>0$. If $ \mathcal{B}$ is commutative and semisimple, then $\varphi$ is continuous.
\end{proposition}
\begin{proof}
Replacing $b$ by $a$ in (\ref{ALM1}), we get
\begin{equation}\label{ALM2}
\Vert \varphi(a^{n+1})-\varphi(a)^{n+1}\Vert \leqslant \delta \Vert a\Vert^{n+1}.
\end{equation}
Thus, $\varphi$ is almost $n$-Jordan homomorphism and so it is continuous by \cite[Theorem 3.4 ]{zivari3}.
\end{proof}

\begin{theorem}
Let $\varphi:\mathcal{A}\longrightarrow \mathcal{B} $ be a linear map such that
\begin{equation}
\Vert \varphi(a^nb)-\varphi(a)^n\varphi(b) \Vert \leqslant \delta (\Vert a\Vert \pm \Vert b \Vert),
\end{equation}
for all $a,b\in \mathcal{A}$ and for some $\delta>0$. Then, $\varphi$ is $(n+1)$-Jordan homomorphism.
\end{theorem}
\begin{proof}
At first, we consider the inequality 
\begin{equation}\label{rr}
\Vert \varphi(a^nb)-\varphi(a)^n\varphi(b) \Vert \leqslant \delta (\Vert a\Vert - \Vert b \Vert),
\end{equation}
for all $a,b\in \mathcal{A}$. Putting $a=b$ in (\ref{rr}), we get $\varphi(a^{n+1})=\varphi(a)^{n+1}$ and so $\varphi$ is $(n+1)$-Jordan homomorphism. Now, assume that

\begin{equation}\label{ALM3}
\Vert \varphi(a^nb)-\varphi(a)^n\varphi(b) \Vert \leqslant \delta (\Vert a\Vert + \Vert b \Vert).
\end{equation}
for all $a,b\in \mathcal{A}$. Replacing $b$ by $a$ in (\ref{ALM3}), we find
\begin{equation}\label{ALM4}
\Vert \varphi(a^{n+1})-\varphi(a)^{n+1}\Vert \leqslant 2\delta \Vert a\Vert.
\end{equation}
for all $a\in \mathcal{A}$. Setting $a=2^mx$, we obtain
\begin{equation}\label{ALM5}
\Vert \varphi(a^{n+1})-\varphi(a)^{n+1}\Vert \leqslant \frac{\delta 2^{m+1}}{2^{m(n+1)}} \Vert a\Vert.
\end{equation}
for all $a\in \mathcal{A}$. Letting $m\rightarrow \infty$, we obtain $\varphi(a^{n+1})=\varphi(a)^{n+1}$ and hence $\varphi$ is $(n+1)$-Jordan homomorphism.
\end{proof}


\section{Pseudo $n$-Jordan homomorphisms}

We commence with a definition.

\begin{definition}
\emph {Let $\mathcal A$ and $\mathcal B$ be rings (algebras), and let $\mathcal B$ be a right [left]
$\mathcal A$-module. We say that a linear mapping $\psi:\mathcal A\longrightarrow\mathcal B$ is a pointwise pseudo $n$-Jordan homomorphism if for each $a\in \mathcal A$ there exists an element $\omega_a\in \mathcal A$ such that $\psi(a^n\omega_a)=\psi(a)^n\cdot \omega_a$ [$(\psi(\omega_a a^n)= \omega_a\cdot\psi(a)^n)$]. We say that $\omega_a$ is a Jordan
coefficient of $\psi$ depended on $a$}.
\end{definition}

It is obviouse that every $n$-Jordan homomorphism from unital Banach algebra $\mathcal{A}$ into $\mathcal{B}$ which is unitary Banach $\mathcal{A}$-module is a pseudo $n$-Jordan homomorphism. Also, as we see that for a pseudo $n$-Jordan homomorphism, there are infinitly many  Jordan coefficient.

Every pseudo $n$-Jordan homomorphism is a pointwise pseudo $n$-Jordan homomorphism. Now, let $\varphi$ be a mixed $n$-Jordan homomorphism such that has a fixed point, say $\omega$. Then, $\varphi$ is a pseudo $n$-Jordan homomorphism with a Jordan coefficient $\omega$.  The following example indicates this fact that the converse is false in general.

\begin{example}\label{eex}
\emph {Let
$$U_2(\mathbb R)=\left\{\left[ \begin{array}{cc}
{a} & {b}  \\
{0} & {c} \\
 \end{array}\right]:\,\, a,b,c\in \mathbb R\right\}$$
 be the algebra of $2\times 2$ matrices with the usual sum and product. Let $\psi:U_2(\mathbb R)\longrightarrow U_2(\mathbb R)$  be a linear
map defined by
$$\psi\left(\left[ \begin{array}{cc}
{a} & {b}  \\
{0} & {c} \\
 \end{array}\right]\right)=\left[ \begin{array}{cc}
{a} & {b}  \\
{0} & {0} \\
 \end{array}\right]$$
 For every $n\in \mathbb N$, we have
 $$\psi\left(\left[ \begin{array}{cc}
{a} & {b}  \\
{0} & {c} \\
 \end{array}\right]^n\right)=\psi\left(\left[ \begin{array}{cc}
{a^n} & {\sum_{k=0}^{n-1}a^{n-k-1}bc^k}  \\
{0} & {c^n} \\
 \end{array}\right]\right)= 
 \left[ \begin{array}{cc}
{a^n} & {\sum_{k=0}^{n-1}a^{n-k-1}bc^k}  \\
{0} & {0} \\
 \end{array}\right]$$
 and
 $$\psi\left(\left[ \begin{array}{cc}
{a} & {b}  \\
{0} & {c} \\
 \end{array}\right]\right)^n=\left[ \begin{array}{cc}
{a} & {b}  \\
{0} & {0} \\
 \end{array}\right]^n=\left[ \begin{array}{cc}
{a^n} & {a^{n-1}b}  \\
{0} & {0} \\
 \end{array}\right]$$
 Thus, $\psi$ is not an $n$-Jordan homomorphism and so it is not mixed $(n-1)$-Jordan homomorphism.. Assume that $t,s\in \mathbb R$. Put $\omega=\left[ \begin{array}{cc}
{s} & {t}  \\
{0} & {0} \\
 \end{array}\right]$. Then,
 $$\psi\left(\left[ \begin{array}{cc}
{a} & {b}  \\
{0} & {c} \\
 \end{array}\right]^n\omega\right)=\psi\left(\left[ \begin{array}{cc}
{a} & {b}  \\
{0} & {c} \\
 \end{array}\right]\right)^n\omega.$$
 This means that $\psi$ is a  pseudo $n$-Jordan homomorphism. On the other hand, for each $a,b,c\in\mathbb R$, take $\omega_{a,b,c}=\left[ \begin{array}{cc}
{s} & {-sc^{1-n}\sum_{k=2}^{n-2}a^{n-k-2}bc^k}  \\
{0} & {0} \\
 \end{array}\right]$. Then,
  $$\psi\left(\omega_{a,b,c}\left[ \begin{array}{cc}
{a} & {b}  \\
{0} & {c} \\
 \end{array}\right]^n\right)=\omega_{a,b,c}\psi\left(\left[ \begin{array}{cc}
{a} & {b}  \\
{0} & {c} \\
 \end{array}\right]\right)^n.$$
     This means that $\psi$ is a pointwise pseudo $n$-Jordan homomorphism. Note that}
     $$\psi\left(\left[ \begin{array}{cc}
{a} & {b}  \\
{0} & {c} \\
 \end{array}\right]^n\omega_{a,b,c}\right)\neq\psi\left(\left[ \begin{array}{cc}
{a} & {b}  \\
{0} & {c} \\
 \end{array}\right]\right)^n\omega_{a,b,c}\,\,\text{ and}\,\,\, \psi\left(\omega\left[ \begin{array}{cc}
{a} & {b}  \\
{0} & {c} \\
 \end{array}\right]^n\right)\neq\omega\psi\left(\left[ \begin{array}{cc}
{a} & {b}  \\
{0} & {c} \\
 \end{array}\right]\right)^n.$$
 
   \end{example}

Here, we remind that the part (3) of \cite[Example 2.2]{Ebadian} is not true. Indeed, it is corrected in Example \ref{eex}.



It is known that every Jordan homomorphism is $n$-Jordan homomorphism \cite{zivari2}. The next example shows that the same result is false for mixed Jordan homomorphisms and pseudo $n$-Jordan homomorphisms.

\begin{example}
\emph {Let $\mathcal{A}$ be a Banach algebra and $f:\mathcal{A}\longrightarrow\mathcal{A} $ be a homomorphism. Define  $\varphi:\mathcal{A}\longrightarrow\mathcal{A}$ by $\varphi(x)=-f(x)$. Then $\varphi$ is a $3$-homomorphism. Thus,
$$
\varphi(x^2a)=\varphi(x)^2\varphi(a),
$$
So $\varphi$ is a mixed Jordan homomorphism, but $\varphi$ is not a mixed $3$-Jordan homomorphism. Suppose that $\varphi$ has a fixed point, say $a$. Hence $\varphi(a)=a$. Thus,
$$
\varphi(x^2a)=\varphi(x)^2\varphi(a)=\varphi(x)^2a,
$$
So $\varphi$ is a pseudo Jordan homomorphism with a Jordan coefficient $a$, but $\varphi$ is not a pseudo $3$-Jordan homomorphism. Note that for all $n\in \mathbb{N}$, $\varphi$ is a pseudo $(2n)$-Jordan homomorphism with a Jordan coefficient $a$, but $\varphi$ is not a pseudo $(2n+1)$-Jordan homomorphism for each $n\in \mathbb{N}$.}
\end{example}

In the sequel,  $\left(\begin{array}{ccccc}
n\\
k\\
\end{array}\right)$ is the binomial coefficient defined for
all $n, k\in \mathbb{N}$ with $n\geq k$ by $n!/(k!(n-k)!)$.

Here and subsequently, let $\mathcal{A}$ and $\mathcal{B}$ be unital Banach algebras and $\mathcal{B}$ be a right $\mathcal{A}$-module. Also, it is assumed that $\varphi$ between unital Banach algebras $\mathcal{A}$ and $\mathcal{B}$ is unital. The following result is the converse of \cite[ Theorem 2.3]{Ebadian}. 
 
\begin{theorem}\label{TT1}
Every unital pseudo $(n+1)$-Jordan homomorphism $\varphi:\mathcal{A}\longrightarrow\mathcal{B} $ with a Jordan coefficient $w$ is a pseudo $n$-Jordan homomorphism.
\end{theorem}
\begin{proof}
We firstly have 
\begin{eqnarray}\label{s1}
\varphi((a+le)^{n+1}w)=(\varphi(a+le))^{n+1}\cdot w, 
\end{eqnarray}
for all $a,b\in \mathcal A$, where $l$ is an integer with $2\leq l\leq n$. It follows from the equality (\ref{s1}) and assumption that
\begin{eqnarray}\label{s2}
\sum_{i=1}^{n}l^i \left(\begin{array}{ccccc}
n+1\\
i\\
\end{array}\right)
[\varphi(a^{i}w)-\varphi(a)^{i}\cdot w]=0,\qquad (2\leq l\leq n),
\end{eqnarray}
for all $a\in \mathcal A$. We can rewrite the equalities in (\ref{s2}) as follows
\begin{eqnarray*}
\left[ \begin{array}{cccc}
{1} & {1} & {\cdots} & {1} \\
{2} & {2^2} & {\cdots} & {2^{n}} \\
{3} & {3^2} & {\cdots} & {3^{n}}\\
{\cdots} & {\cdots} & {\cdots} & {\cdots} \\
{n} & {n^2} & {\cdots} & {n^{n}} \\
 \end{array} \right] 
 \left[ \begin{array}{cccc}
 {\Gamma_1(a,w)} \\
{\Gamma_2(a,w)} \\
{\Gamma_3(a,w)} \\
{\cdots}  \\
{\Gamma_n(a,w)} \\
 \end{array} \right]
=\left[ \begin{array}{cccc}
{0} \\
{0} \\
{0} \\
{\cdots}  \\
{0} \\
 \end{array} \right]
\end{eqnarray*}
for all $a\in \mathcal A$, where $\Gamma_i(a,w)=\left(\begin{array}{ccccc}
n+1\\
i\\
\end{array}\right)[\varphi(a^{i}w)-\varphi(a)^{i}\cdot w]$ for all $1\leq i\leq n$. It is shown in \cite[Lemma 2.1]{Bodaghi} that the above square matrix is invertible. 
This implies that $\Gamma_i(a,w)=0$ for all $1\leq i\leq n$ and all $a\in \mathcal A$. In particular, $\Gamma_n(a,w)=0$. This means that $\varphi$ is a pseudo $n$-Jordan homomorphism.
\end{proof}

The next Corollary follows from Theorem \ref{TT1}.
\begin{corollary}\label{C1}
Every unital $(n+1)$-Jordan homomorphism $\varphi:\mathcal{A}\longrightarrow\mathcal{B} $ is an $n$-Jordan homomorphism.
\end{corollary}


\begin{example}
\emph{Let}
\begin{equation*}
\mathcal{A}= \left\{
 \begin{bmatrix}
0 & a & b \\
0 &  0 & c\\
0 &  0 & 0
\end{bmatrix}
:\ \ \  a,b,c\in\mathbb{R}
\right\},
\end{equation*}
\emph{and define $\varphi: \mathcal{A} \longrightarrow \mathcal{A}$ by}
$$
\varphi\left(\begin{bmatrix}
0 & a & b \\
0 &  0 & c\\
0 &  0 & 0
\end{bmatrix}\right)
=
\begin{bmatrix}
0 & a & 0 \\
0 &  0 & c\\
0 &  0 & 0
\end{bmatrix},
$$
\emph{Then, for all $X\in \mathcal{A}$, $\varphi(X^2)\neq\varphi(X)^2$. Hence, $\varphi$ is not Jordan homomorphism, but for all $n\geq 3$ and all $X\in \mathcal{A}$, we have $\varphi(X^n)=\varphi(X)^n.$ Therefore, $\varphi$ is $n$-Jordan homomorphism for all $n\geq 3$. Assume that $s,t,r\in \mathbb{R}$ is arbitrary. Put}
$$
w=
\begin{bmatrix}
0 & s & t \\
0 & 0 & r \\
0 & 0 & 0
\end{bmatrix}.
$$
\emph{Hence, for all $n\in \mathbb{N}$ and $X\in \mathcal{A}$, we get $\varphi(X^nw)=\varphi(X)^nw$. Therefore, $\varphi$ is a pseudo $n$-Jordan homomorphism. In other words, we showed the condition that being unital for Banach algebras $\mathcal{A}$ and $\mathcal{B}$ in Corollary \ref{C1} is essential.}
\end{example}


Suppose that $\mathcal{A}$ is a Banach algebra and $M$ is an $\mathcal{A}$-module. Let $w\in \mathcal{A}$. Then $w$ is called a left (right) separating point of $M$ if the condition $wx=0$ ($xw=0$) for $x\in M$ implies that $x=0$ \cite{Lu}.\\

In the following result, under certain conditions, we prove that each pseudo Jordan homomorphism is pseudo $n$-Jordan homomorphism.
\begin{theorem}\label{TT2}
Let $\varphi:\mathcal{A}\longrightarrow\mathcal{B} $ be a unital pseudo Jordan homomorphism with a Jordan coefficient $w$ such that $w$ is a right separating point of $\mathcal{B}$. Then
\begin{enumerate}
\item [(i)]
 $\varphi$ is $n$-Jordan homomorphism.

\item [(ii)]
$\varphi$ is pseudo $n$-Jordan homomorphism.
\end{enumerate}
\end{theorem}
\begin{proof}
Assume that $\varphi$ is a pseudo Jordan homomorphism, then
\begin{equation}\label{Q1}
\varphi(a^2w)=\varphi(a)^2\cdot w,
\end{equation}
for all $a\in \mathcal{A}$. Replacing $a$ by $a+b$ in (\ref{Q1}), we get
\begin{equation}\label{Q2}
\varphi[(ab+ba)w]=[\varphi(a)\varphi(b)+\varphi(b)\varphi(a)]\cdot w.
\end{equation}
Since $\varphi$ is unital, one can show that
\begin{equation}\label{Q3}
\varphi[(ab+ba)w]=\varphi(ab+ba)\cdot w.
\end{equation}
It follows from (\ref{Q2}) and (\ref{Q3}) that
\begin{equation}
\big(\varphi(ab+ba)-[\varphi(a)\varphi(b)+\varphi(b)\varphi(a)]\big)\cdot w=0
\end{equation}
Since $w$ is a right separating point of $\mathcal{B}$, we get
$$
\varphi(ab+ba)=\varphi(a)\varphi(b)+\varphi(b)\varphi(a).
$$
Thus, $\varphi$ is a Jordan homomorphism. Now by Lemma 2.6 of \cite{zivari2}, $\varphi$ is $n$-Jordan homomorphism. The Part (2) follows from above Lemma and (1).
\end{proof}
From Theorem \ref{TT2},  \cite[Theorem 2.3]{Bodaghi} and  \cite[Corollary 2.5]{An}, we have the following trivial consequence.
\begin{corollary}
By hypotheses of Theorem \ref{TT2}, $\varphi$ is $n$-homomorphism if either
\begin{enumerate}
\item [(i)]
 $\mathcal{A}$ and $\mathcal{B}$ are commutative, or

\item [(ii)]
$\mathcal{B}$ is commutative and semisimple.
\end{enumerate}
\end{corollary}


\begin{corollary}
Let $\varphi:\mathcal{A}\longrightarrow\mathcal{B} $ be a unital pseudo $n$-Jordan homomorphism with a Jordan coefficient $w$ such that $w$ is a right separating point of $\mathcal{B}$. Then $\varphi$ is continuous with each of the following conditions.
\begin{enumerate}
\item [(i)]
 $\mathcal{B}$ is semisimple and commutative.

\item [(ii)]
$\mathcal{B}$ is semisimple and $\varphi$ is surjective.

\item [(iii)]
$\mathcal{B}$ is $C^*$-algebra and $\varphi$ is surjective.
\end{enumerate}
\end{corollary}
\begin{proof}
By Theorem \ref{TT1},  $\varphi$ is pseudo Jordan homomorphism and by Theorem \ref{TT2} it is $n$-Jordan homomorphism. Thus, the result follows from Corollaries 2.9, 2.10 and 2.11 of \cite{zivari3}.
\end{proof}
\begin{theorem}
Let $\varphi:\mathcal{A}\longrightarrow\mathcal{B}$ be a unital pseudo $3$-Jordan homomorphism with a Jordan coefficient $w$ such that $w$ is a right separating point of $\mathcal{B}$. Suppose that $\mathcal{B}$ is commutative and
\begin{equation}\label{Q4}
\varphi(abcw)=\varphi(acbw),\    \    \   \   (a,b,c\in \mathcal{A}).
\end{equation}
Then, $\varphi$ is homomorphism.
\end{theorem}
\begin{proof}
Assume that $e$ is an unit element of $\mathcal{A}$, letting $a=e$ in (\ref{Q4}), we get
$$
\varphi(bcw-cbw)=0,
$$
for all $b,c\in \mathcal{A}$. Thus, $\varphi((ab)cw)=\varphi(c(ab)w)=\varphi(c(ba)w)$ and hence
$$
\varphi(a(bc)w)=\varphi((bc)aw)=\varphi(b(ca)w)=\varphi(b(ac)w).
$$
That is
\begin{equation}\label{Q5}
\varphi(abcw)=\varphi(xyzw),
\end{equation}
whenever $(x, y, z)$ is a permutation of $(a, b, c)$. Since $\varphi$ is a pseudo $3$-Jordan homomorphism, $\varphi(a^3w)=\varphi(a)^3\cdot w$, for all $a\in \mathcal{A}$. Replacing $a$ by $a+b$, we get
\begin{equation}\label{Q6}
\varphi[(ab^2+b^2a+a^2b+ba^2+aba+bab)w]=[3\varphi(a)\varphi(b)^2+3\varphi(a)^2\varphi(b)]\cdot w.
\end{equation}
Interchanging $b$ by $-b$ in (\ref{Q6}), we obtain
\begin{equation}\label{Q7}
\varphi[(ab^2+b^2a-a^2b-ba^2-aba+bab)w]=[3\varphi(a)\varphi(b)^2-3\varphi(a)^2\varphi(b)]\cdot w.
\end{equation}
The relations  (\ref{Q6}), and (\ref{Q7}) imply that
\begin{equation}\label{Q8}
\varphi[(ab^2+b^2a+bab)w]=[3\varphi(a)\varphi(b)^2]\cdot w.
\end{equation}
Replacing $b$ by $b-c$ in (\ref{Q8}), we deduce
\begin{equation}\label{Q9}
\varphi[(abc+acb+bac+bca+cab+cba)w]=[6\varphi(a)\varphi(b)\varphi(c)]\cdot w.
\end{equation}
It follows from (\ref{Q5}) and (\ref{Q9})  that
\begin{equation}\label{Q10}
\varphi(abcw)=[\varphi(a)\varphi(b)\varphi(c)]\cdot w,
\end{equation}
for all $a,b,c\in \mathcal{A}$. Put $c=e$ in (\ref{Q10}), we arrive at
$$
(\varphi(ab)-[\varphi(a)\varphi(b)])\cdot w=0.
$$
Since $w$ is a right separating point of $\mathcal{B}$, we get
$$
\varphi(ab)=\varphi(a)\varphi(b), \   \    \   a,b\in \mathcal{A}.
$$
Thus, $\varphi$ is a homomorphism.
\end{proof}


\end{document}